\documentclass[a4paper,11pt]{article}

\usepackage{hyperref}
\usepackage{amsmath}
\usepackage{amsthm}
\usepackage{amsfonts}
\usepackage{amssymb}
\usepackage{verbatim}
\usepackage{tikz}
\usepackage[bottom,symbol]{footmisc}
\usepackage{geometry}
\title{Tur\'an $H$-densities for 3-graphs}

\author{Victor Falgas-Ravry \thanks{School of Mathematical Sciences, Queen Mary, University of London, Mile End Road, London E1 4NS, U.K. Email: {\tt v.falgas-ravry@qmul.ac.uk}.}
\and Emil R. Vaughan \thanks{School of Electronic Engineering and Computer Science, Queen Mary, University of London, Mile End Road, London E1 4NS, U.K. Email: {\tt e.vaughan@qmul.ac.uk}. Supported by EPSRC grant EP/H016015/1.}}

\newtheorem{theorem}{Theorem}
\newtheorem{corollary}[theorem]{Corollary}
\newtheorem{lemma}[theorem]{Lemma}
\newtheorem{proposition}[theorem]{Proposition}

\newtheorem{conjecture}{Conjecture}
\newtheorem{question}[conjecture]{Question}

\DeclareMathOperator{\ex}{ex}
\providecommand{\size}[1]{\left|#1\right|}
\newcommand{\Text}[1]{\text{\textnormal{#1}}}

\begin{document}
\bibliographystyle{plain}
\maketitle

\begin{abstract}
Given an $r$-graph $H$ on $h$ vertices, and a family $\mathcal{F}$ of forbidden subgraphs, we define $\ex_{H}(n, \mathcal{F})$ to be the maximum number of induced copies of $H$ in an $\mathcal{F}$-free $r$-graph on $n$ vertices. Then the \emph{Tur\'an $H$-density} of $\mathcal{F}$ is the limit 
\[\pi_{H}(\mathcal{F})= \lim_{n\rightarrow \infty}\ex_{H}(n, \mathcal{F})/\binom{n}{h}. \]
This generalises the notions of \emph{Tur\'an density} (when $H$ is an $r$-edge), and \emph{inducibility} (when $\mathcal{F}$ is empty). Although problems of this kind have received some attention, very few results are known.

We use Razborov's semi-definite method to investigate Tur\'an $H$-densities for $3$-graphs. In particular, we show that
\[\pi_{K_4^-}(K_4) = 16/27,\]
with Tur\'an's construction being optimal. We prove a result in a similar flavour for $K_5$ and make a general conjecture on the value of $\pi_{K_t^-}(K_t)$. We also establish that
\[\pi_{4.2}(\emptyset)=3/4,\]
where $4.2$ denotes the $3$-graph on $4$ vertices with exactly $2$ edges. The lower bound in this case comes from a random geometric construction strikingly different from previous known extremal examples in $3$-graph theory. We give a number of other results and conjectures for $3$-graphs, and in addition consider the inducibility of certain directed graphs. Let $\vec{S}_k$ be the \emph{out-star} on $k$ vertices; i.e{.} the star on $k$ vertices with all $k-1$ edges oriented away from the centre. We show that 
\[\pi_{\vec{S}_3}(\emptyset)=2\sqrt3-3,\]
with an iterated blow-up construction being extremal. This is related to a conjecture of Mubayi and R\"odl on the Tur\'an density of the 3-graph $C_5$. We also determine $\pi_{\vec{S}_k}(\emptyset)$ when $k=4$, and conjecture its value for general $k$.
\end{abstract}

\section{Introduction}
\subsection{Basic notation and definitions}
Given $n \in \mathbb{N}$, write $[n]$ for the integer interval $\{1,2, \dots, n\}$. Let $r \in \mathbb{N}$. An \emph{$r$-graph} or \emph{$r$-uniform hypergraph} $G$ is a pair $G=(V,E)$, where $V=V(G)$ is a set of \emph{vertices} and $E=E(G) \subseteq V^{(r)}=\{A \subseteq V: \ \vert A \vert = r\}$ is a set of \emph{$r$-edges}. We shall often write $x_1x_2\cdots x_r$ as a short-hand for the $r$-edge $\{x_1,x_2, \dots, x_r\}$.

Given a family of $r$-graphs $\mathcal{F}$, we say that $G$ is $\mathcal{F}$-free if it contains no member of $\mathcal{F}$ as a subgraph. A classical aim of extremal hypergraph theory is to determine the maximum number of $r$-edges that an $\mathcal{F}$-free $r$-graph on $n$ vertices may contain. We call the corresponding function of $n$ the \emph{T\'uran number} of $\mathcal{F}$, and denote it by
\[\ex(n, \mathcal{F}) = \max\left\{ \size{E(G)} \ : \text{ $G$ is $\mathcal{F}$-free}, \ \size{V(G)}=n\right\}.\]

In this paper we shall be concerned with the following generalisation of the Tur\'an number. Given an $r$-graph $H$ on $h$ vertices, and an $r$-graph $G$ on $n \geq h$ vertices, let $e_H(G)$ denote the number of $h$-sets from $V(G)$ that induce a copy of $H$ in $G$. (So for example if $H$ is an $r$-edge, then $e_H(G)$ counts the number of edges in $G$.) Then, given a family of forbidden $r$-graphs $\mathcal{F}$, we define the \emph{Tur\'an $H$-number} of $\mathcal{F}$, denoted $\ex_{H}(n, \mathcal{F})$, to be the maximum number of induced copies of $H$ that an $\mathcal{F}$-free $r$-graph on $n$ vertices may contain:
\[\ex_{H}(n, \mathcal{F}) = \max\left\{e_H(G): \text{ $G$ is $\mathcal{F}$-free}, \ \size{V(G)}=n\right\}.\]
In general, the Tur\'an $H$-number is, like the usual Tur\'an number, hard to determine, and we are interested instead in the asymptotic proportion of $h$-vertex subsets that induce a copy of $H$. The following is well-known.

\begin{proposition}\label{averaging}
Let $\mathcal{F}$ be a family of $r$-graphs and let $H$ be an $r$-graph on $h$ vertices. Then the limit
\[\pi_{H}(\mathcal{F})= \lim_{n \rightarrow \infty} \ex_{H}(n, \mathcal{F})/\binom{n}{h} \]
exists.
\end{proposition}
\begin{proof}
For $n\geq h$, it follows by averaging over $n$-vertex subsets that
\[\ex_{H}(n+1, \mathcal{F})/\binom{n+1}{h}\leq \ex_{H}(n, \mathcal{F})/\binom{n}{h}.\]
Thus the sequence $\ex_H(n, \mathcal{F})/\binom{n}{h}$ is nonincreasing, and because it is bounded below (e.g{.} by $0$), it is convergent. 
\end{proof}

We call $\pi_{H}(\mathcal{F})$ the \emph{Tur\'an $H$-density} of $\mathcal{F}$. In the case where $H$ is the $r$-graph on $r$ vertices with a single edge, we recover the classical \emph{Tur\'an density}, $\pi(\mathcal{F})$.

It is easy to see that Proposition~\ref{averaging} and the definitions of $\ex_H(n, \mathcal{F})$ and $\pi_{H}(\mathcal{F})$ when $H$ and $\mathcal{F}$ consist of $r$-graphs could just as well have been made in the setting of directed $r$-graphs. We let our definitions carry over \emph{mutatis mutandis}.

In this paper, we shall mainly investigate 3-graphs, although we shall make a digression into directed $2$-graphs in Section~3.

	\subsection{Previous work on inducibility}

When $\mathcal{F}=\emptyset$, $\pi_{H}(\emptyset)$ is known as the \emph{inducibility} of $H$. The inducibility of 2-graphs was first investigated by Pippenger and Golumbic~\cite{PG75} and later by Exoo~\cite{E86}. Motivated by certain questions in Ramsey Theory, Exoo proved some general bounds on $\pi_{H}(\emptyset)$ as well as giving some constructions for small $H$ with $\size{V(H)} \leq 4$. Bollob\'as, Nara and Tachibana~\cite{BNT86} then proved that $\pi_{K_{t,t}}(\emptyset)=(2t)!/2^t{(t!)}^2$, where $K_{t,t}$ is the balanced complete bipartite graph on $2t$ vertices, $K_{t,t}= ([2t], \{\{ij\}:\ i \leq t < j\})$. What is more, they determined $\ex_{K_{t,t}}(n, \emptyset)$ exactly, with the optimal construction a balanced complete bipartite graph. More generally, Brown and Sidorenko~\cite{BS94} showed that if $H$ is complete bipartite then the graphs attaining the Tur\'an $H$-number may be chosen to be themselves complete bipartite.

Given a graph $H$  and an integer $b\geq 1$, the \emph{(balanced) $b$-blow-up} of $H$, denoted $H(b)$, is the graph on $b \size{V(H)}$ vertices obtained by taking for every vertex $x \in V(H)$  a set of $b$ vertices $x_1,x_2, \dots, x_b$ and putting an edge between $x_i$ and $y_j$ if and only if $xy \in E(H)$. Bollob\'as, Egawa, Harris and Jin~\cite{BEHJ95} proved that for all $t \in \mathbb{N}$ and all $b$ sufficiently large, the Tur\'an $K_t(b)$-number $\ex_{K_t(b)}(n, \emptyset)$  is attained by balanced blow-ups of $K_t$. This was recently generalised in an asymptotic sense by Hatami, Hirst and Norine~\cite{HHN11} who proved that for any graph $H$ and for all $b$ sufficiently large, the Tur\'an $H(b)$-density is given by considering the `limit' of balanced blow-ups of $H$. Their proof relied on the use of weighted graphs.

Finally, several $H$-density results for small $H$ were obtained this year by Grzesik~\cite{G11}, Hatami, Hladk\'y, Kr\'al, Norine and Razborov~\cite{HHKNR11a}, Hirst~\cite{H11} and Sperfeld~\cite{S11}, all using the semi-definite method of Razborov~\cite{R07}. Grzesik~\cite{G11}, and independently Hatami, Hladk\'y, Kr\'al, Norine and Razborov~\cite{HHKNR11a}, proved an old conjecture of Erd\H os~\cite{E84} that the number of (induced) copies of the $5$-cycle $C_5=([5],\{12,23,34,45,51\})$ in a triangle-free graph on $n$ vertices is at most $(n/5)^5$. This bound is attained by a balanced blow-up of $C_5$, thus establishing that
\[\pi_{C_5}(K_3)=24/625.\]
 To describe the other two sets of results, we need to make some more definitions. Let
\[K_{1,1,2}=([4], \{12,13,14, 23,24\}),\  \Text{paw}=([4], \{12,23, 31, 14\})\]
and
\[\vec{C}_3=([3], \{\vec{12}, \vec{23}, \vec{31}\}), \ \vec{K}_2\sqcup E_1= ([3], \{\vec{12}\}).\]
Then Hirst showed that
\[\pi_{K_{1,1,2}}(\emptyset)=72/125, \  \pi_{\text{paw}}(\emptyset) =3/8,\]
with extremal configurations a balanced blow-up of $K_5$ and the complement of a balanced blow-up of $([4],\{12,34\})$ respectively. Sperfeld proved
\[\pi_{\vec{C}_3}(\emptyset)=1/4, \ \pi_{\vec{K}_2\sqcup E_1}(\emptyset)=3/4,\]
with extremal configurations a random tournament on $n$ vertices and the disjoint union of two tournaments on $n/2$ vertices respectively.

	\subsection{Flag algebras and Flagmatic} \label{flagmaticsec}

Similarly to the works cited above~\cite{G11, HHKNR11a, H11, S11}, the upper bounds on Tur\'an $H$-densities we present in this paper have been obtained using the semi-definite method of Razborov~\cite{R07}. A by-product of the theory of flag algebras, the semi-definite method gives us a systematic way of proving linear inequalities between subgraph densities. It has recently been used in a variety of contexts and has yielded many new results and improved bounds. (See e.g{.} \cite{BT10, BT11, FRV11, G11, HHKNR11a, HHKNR11b, H11, KMS11, R10, R11, S11}.)

While it is clearly a powerful and useful tool in extremal combinatorics, the semi-definite method requires its users to overcome two barriers. First of all, a presentation of the method is usually given in the language of flag algebras, quantum graphs or graphons, which, while not impenetrable, is certainly forbidding at first. Second, the method involves numerous small computations, the enumeration of large graph families and optimisation of the entries of large positive semi-definite matrices; none of which can practically be done by hand. The assistance of a computer program is therefore necessary to use the semi-definite method in any nontrivial fashion.

In our earlier paper~\cite{FRV11}, we sought to remove these two obstacles by giving an elementary presentation of the semi-definite method from the point of view of extremal combinatorics, stripping it away from the more general framework of flag algebras, and by releasing `Flagmatic', an open-source implementation of Razborov's semi-definite method.

Additionally, in an effort to avoid having large matrices and lists of graphs cluttering the main body of the paper, we have used Flagmatic to produce certificates of our results. These certificates, along with Flagmatic, can be downloaded from our website:

\begin{quote}
\url{http://www.maths.qmul.ac.uk/~ev/flagmatic/}
\end{quote}

The certificates are also given in the ancillary files section of our arXiv submission. The certificates are in a straight-forward human-readable format, which is documented in our previous paper \cite{FRV11}. The website also contains an independent checker program called \verb|inspect_certificate.py|, which can be used to examine the certificates and help verify our proofs.

We shall not repeat here our introduction to the semi-definite method, nor our discussion of certificates and checker programs, but refer the reader back to~\cite{FRV11} for details and use Flagmatic as a `black box' for the remainder of this paper. 

Finally, let us note that some information on extremal constructions can sometimes be extracted from proofs via the semi-definite method. We address this, and in particular the issue of stability, in a forthcoming paper~\cite{FRV12}.

	\subsection{Contents and structure of the paper}

Let us define formally the $3$-graphs that we study in this paper. First of all, we have the complete 3-graph on $4$ vertices, $K_4$, also known as the \emph{tetrahedron}. We shall also be interested in $K_4^-$, the unique (up to isomorphism) $3$-graph on $4$ vertices with exactly $3$ edges, and in the \emph{(strong) $5$-cycle}, $C_5=([5], \{123,234,345,451,512\})$. Let also $K_t$ denote the complete $3$-graph on $t$ vertices and $K_t^-$ the $3$-graph obtained from $K_t$ by deleting a $3$-edge, and let $H_6$ be the $3$-graph obtained from $C_5$ by adding  a new vertex labelled `$6$' to the vertex set  and adding the following five edges: $136,356, 526,246,416$.

A $3$-graph is said to have \emph{independent neighbourhoods} if for any pair of distinct vertices $x,y$, the joint neighbourhood of $x,y$, \[\Gamma_{xy}=\{z : \text{ $xyz$ is an edge}\}\]
is an independent set. Having independent neighbourhoods is easily seen to be equivalent to not containing the graph $F_{3,2}=([5], \{123,124,125,345\})$ as a subgraph.

Finally, following the notation used by Flagmatic, we write $m.k$ for the collection of all $3$-graphs on $m$ vertices spanning exactly $k$ edges, up to isomorphism. For example, $4.3=\{K_4^-\}$.

Our exact results for Tur\'an $H$-densities of $3$-graphs are listed in the following table:

\begin{center}
\begin{tabular}{l p{6.5cm}}
Result & Extremal construction \\ \hline
$\pi_{K_4^-}(K_4) = 16/27$ & Tur\'an's construction: balanced blow-up of $([3], \{112,223,331, 123\})$.\\

$\pi_{4.2}(\emptyset) = 3/4$ & Random geometric construction; see Theorem~\ref{4.2}.\\

$\pi_{4.2}(C_5, F_{3,2}) = 9/16$ & Balanced blow-up of $K_4$.\\

$\pi_{4.2}(K_4^-, F_{3,2}) = 5/9$ & Balanced blow-up of $H_6$.\\

$\pi_{4.2}(K_4^-,C_5, F_{3,2}) = 4/9$ & Balanced blow-up of a 3-edge.\\

$\pi_{K_4}(F_{3,2}) = 3/32$ & Balanced blow-up of $K_4$.\\

$\pi_{K_4^-}(F_{3,2}) = 27/64$ & Unbalanced blow-up of $([2], \{112\})$.\\

$\pi_{5.6}(\emptyset) = 20/27$ &  Balanced blow-up of the $3$-graph $([3], \{112, 221, 223, 332, 113, 331\})$. \\
$\pi_{5.7}(\emptyset) = 20/27$ & Balanced blow-up of the $3$-graph $([3], \{111, 222, 333, 112, 223, 331, 123\})$. \\
$\pi_{5.9}(\emptyset) = 5/8$ & Balanced complete bipartite 3-graph.\\
\hline
\end{tabular}
\end{center}

In addition, we prove two inducibility results for directed graphs. 
We define the \emph{out-star} of order $k$ to be the directed graph
\[\vec{S}_k= ([k], \{\vec{1i}: \ i \in [k]\setminus\{1\}\}). \]
We prove that
\[\pi_{\vec{S}_3}(\emptyset)= 2\sqrt{3}-3,\] 
with the extremal construction being an unbalanced blow-up of $\vec{S}_2$, iterated inside the part corresponding to the vertex labelled $2$. (Here `iterated' just means: repeat the construction inside the vertices that were allocated to part $2$ after each iteration of the construction, until you run out of vertices.) Sperfeld \cite{S11} previously gave bounds for this problem.

This result is interesting to us for two reasons: first of all, this directed $2$-graph problem has a somewhat close and unexpected relation to the Tur\'an problem of maximising the number of $3$-edges in a $C_5$-free $3$-graph. Second, we believe this is the first `simple' instance for which it can be shown that an iterated blow-up construction is extremal. (We elaborate on this in Section~3.)

While it is not directly relevant to $3$-graphs, which are the main focus of this paper, we also determine $\pi_{\vec{S}_4}(\emptyset)$ and make a conjecture regarding the value of $\pi_{\vec{S}_k}(\emptyset)$ for all $k \ge 5$.

Our paper is structured as follows: in Section~2 we present our $3$-graph results. Section~2.1 deals with the case where we forbid $K_4$ and other complete graphs, while Sections~2.2, 2.3 and 2.4 are concerned with the cases where we forbid $C_5$, $K_4^-$ and both $C_5$ and $K_4^-$ respectively. In Section~2.5 we consider $3$-graphs with the independent neighbourhood property, and Section~2.6 gathers our results on inducibilities of $3$-graphs, in particular our proof that $\pi_{4.2}(\emptyset)=3/4$. Finally, in Section~3 we move on to consider directed $2$-graphs and discuss the relation between $\pi_{\vec{S}_3}(\emptyset)$ and a conjecture of Mubayi and R\"odl regarding the Tur\'an density of the $3$-graph $C_5$.

As previously mentioned, the certificates for all the results are available on the Flagmatic website, and in the ancillary files of our arXiv submission. Each certificate has a unique filename, which is given in the following table:

\begin{center} {\small
\begin{tabular}{p{2.5cm}p{3.1cm}|p{2.5cm}p{2.7cm}}
Result & Certificate & Result & Certificate \\
\hline
Theorem~\ref{k4-nok4} & \verb|k4max43.js| & Theorem~\ref{4.2} & \verb|max42.js| \\
Proposition~\ref{k4-noc5} & \verb|c5max43.js| & Proposition~\ref{k4-inducibility} & \verb|max43.js|
and \verb|41max43.js| \\
Proposition~\ref{4.2noc5} & \verb|c5max42.js| & Theorem~\ref{5.6} & \verb|max56.js| \\
Theorem~\ref{4.2noc5f32} & \verb|c5f32max42.js| & Theorem~\ref{5.7} & \verb|max57.js| \\
Proposition~\ref{4.2nok4-} & \verb|k4-max42.js| & Theorem~\ref{5.9} & \verb|max59.js| \\
Theorem~\ref{4.2nok4-nof32} & \verb|k4-f32max42.js| & Proposition~\ref{f32} & \verb|maxf32.js| \\
Proposition~\ref{4.2nok4-c5} & \verb|k4-c5max42.js| & Proposition~\ref{c5} & \verb|maxc5.js| \\
Theorem~\ref{4.2nok4-c5f32} & \verb|k4-c5f32max42.js| & Theorem~\ref{1213} & \verb|maxs3.js| \\
Theorem~\ref{k4-nof32} & \verb|f32max43.js| & Theorem~\ref{121314} & \verb|maxs4.js| \\
Theorem~\ref{k4nof32} & \verb|f32max44.js| & & \\
Proposition~\ref{4.2nof32} & \verb|f32max42.js| and \verb|f32max41.js| & & \\
\hline
\end{tabular} }
\end{center}

\section{Main results}

	\subsection{Forbidding $K_4$}

The problem of determining the Tur\'an density of the complete $3$-graph on $4$ vertices, $K_4$, has been open for more than sixty years. Tur\'an conjectured that the answer is $5/9$, with the lower bound coming from a balanced blow-up of $([3], \{112,223,331, 123\})$.

\begin{conjecture}[Tur\'an] \label{turanconj}
\[\pi(K_4)=5/9.\]
\end{conjecture}

Many other non-isomorphic $K_4$-free constructions with asymptotic edge-density $5/9$ have since been found~\cite{B83, FDF88, F08, K82}, so that if Tur\'an's conjecture is true, there is no stable extremal configuration and a proof is likely to be very hard.

Razborov observed that Tur\'an's original construction is the only one known in which no $4$-set spans exactly one $3$-edge. Adding in this restriction, he found that he could use the semi-definite method to prove a weaker form of Tur\'an's conjecture:

\begin{theorem}[Razborov~\cite{R10}]
\[\pi(K_4, \Text{ induced }4.1)=5/9.\]
\end{theorem}

What is more, Pikhurko~\cite{P11a} showed that  Tur\'an's construction is the unique, stable extremal configuration for this problem. We can show that in fact what Tur\'an's construction does is to maximise the $K_4^-$-density in $K_4$-free $3$-graphs; this can be thought of as the most natural weakening of Tur\'an's conjecture.

\begin{theorem}\label{k4-nok4}
\[\pi_{K_4^-}(K_4)=16/27\]
\end{theorem}

\begin{proof}
The upper bound is from a flag algebra calculation using Flagmatic (see Section~\ref{flagmaticsec} for how to obtain a certificate). The lower bound is from Tur\'an's construction, a balanced blow-up of $([3], \{112,223,331, 123\})$.
\end{proof}

In addition, by essentially mimicking Pikhurko's argument, it is possible to show that any $K_4$-free $3$-graph with $K_4^-$-density `close' to $16/27$ is `close' to Tur\'an's construction in the \emph{edit distance}. That is, one can make it into a copy of Tur\'an's construction by changing `few' edges. We address this, and the more general issue of obtaining stability from proofs via the semi-definite method, in a forthcoming note~\cite{FRV12}.

Having established that $\pi_{K_4^-}(K_4)=16/27$, can we say anything about $\pi_{K_5^-}(K_5)$? In Section~2.4 we give a result, Theorem~\ref{5.9}, that implies $\pi_{K_5^-}(K_5)=5/8$, with the lower bound coming from a complete balanced bipartite $3$-graph. More generally, we believe we know what the value of $\pi_{K_t^-}(K_t)$ should be.

Define a sequence $(H_t)_{t\geq 2}$ of degenerate $3$-graphs on $t$ vertices as follows. Let 
\[H_2=\left([2], \{111, 222, 112, 221\}\right),\]
and 
\[H_3=\left([3], \{111, 222, 333, 112, 223, 331\}\right).\]  
Now for $t\geq 4$, define $H_t$ by adding vertices $t-1$ and $t$ to $H_{t-2}$, together with the edges
\[ (t-1)(t-1)(t-1), \ ttt, \ (t-1)(t-1)t, \ (t-1)tt. \]
Then let $G_t(n)$ denote the complement of a balanced blow-up of $H_{t-1}$ on $n$ vertices. This construction is due to Keevash and Mubayi, and is well-known (see for example Keevash~\cite{Keevash survey}) to be $K_{t}$-free.

\begin{conjecture}\label{superconjecture}
$G_t(n)$ is the unique  (up to isomorphism) $3$-graph with $\ex(n, K_t)$ edges and $\ex_{K_t^-}(n, K_t)$ induced copies of $K_t^-$. 
\end{conjecture}

It is easy to work out that $G_t(n)$ has edge-density
\[1-\frac4{(t-1)^2}+o(1),\]
and a slightly more involved calculation shows that its $K_t^-$-density is
\[\dfrac{t!}{(t-1)^{t-1}} \dfrac{ 2^{(t-1)/2}}{3}+o(1) \]
if $t$ is odd, and
\[ \dfrac{t!}{(t-1)^t} \dfrac{(5t-8)}{3} 2^{(t-6)/2}+o(1) \]
if $t$ is even.
Note that for $t=4$ and $t=5$ this agrees with Theorems~\ref{k4-nok4} and~\ref{5.9} respectively.

	\subsection{Forbidding $C_5$}

Mubayi and R\"odl~\cite{MR02} studied the Tur\'an density problem for $C_5$, and came up with the following ingenious construction. Partition the vertex set into two parts $A$ and $B$ with $\size{A} \approx \sqrt{3} \size{B}$, and add all edges that have two vertices in $A$ and one vertex in $B$, and then iterate inside $B$. This can be described succinctly as an unbalanced blow-up of the (degenerate) $3$-graph $([2], \{112\})$, iterated inside part $2$. We leave it as an exercise for the reader to verify that this is indeed $C_5$-free. Mubayi and R\"odl conjectured that this construction is best possible, and recent applications of the semi-definite method~\cite{FRV11, R10} have provided strong evidence in that direction.

\begin{conjecture}[Mubayi, R\"odl~\cite{MR02}]\label{c5conj}
\[\pi(C_5)= 2\sqrt{3}-3.\]
\end{conjecture}

Observe now that Mubayi and R\"odl's construction avoids $K_4$ as well as $C_5$. If their conjecture is true, then $\pi(C_5)=\pi(C_5, K_4)$. We would thus expect their construction to also maximise the number of copies of $K_4^-$ in a $C_5$-free $3$-graph. This appears to be the case, with a minor caveat: the construction is the right one, but the weights we place on each part need to be adjusted slightly.

\begin{proposition}\label{k4-noc5}
\[0.423570 < \alpha \leq \pi_{K_4^-}(C_5) < 0.423592,\]
where $\alpha$ is the maximum value of 
\[f(x)=\frac{4x(1-x)^3}{1-x^4},\]
in the interval $[0,1]$, which, by solving a cubic equation, can be computed explicitly to be
\[\alpha = 4-6\left((\sqrt{2}+1)^{1/3}-(\sqrt{2}-1)^{1/3}\right).\]
\end{proposition}

\begin{proof} The upper bound is from a flag algebra calculation using Flagmatic (see Section~\ref{flagmaticsec} for how to obtain a certificate). The lower bound is from a blow-up of $([2], \{112\})$, with proportion $x$ of the vertices placed inside part $2$, iterated inside part $2$. The function $f(x)$ then calculates exactly the asymptotic density of $K_4^-$ in such a construction. The sign of the derivative of $f$ is determined by the product of a cubic and a linear factor. Performing the required calculus, the maximum of $f$ can then be determined in closed form.
\end{proof}

Note that the maximum of $f$ occurs at a cubic irrational, and not at a quadratic irrational as happens when we maximise the number of $3$-edges. What is more, we place proportion approximately $0.366025$ (i.e{.} a little more than $1/3$) of the vertices inside part $B$ when maximising the edge-density; and this drops down to approximately $0.253077$ (i.e{.} a little more than $1/4$) when maximising the $K_4^-$ density. This is to be expected; in the first case we want an average $3$-set to have about one vertex in part $B$, while in the latter case we want an average $4$-set to have about one vertex in part $B$.

We conjecture that the lower bound in Proposition~\ref{k4-noc5} is tight:

\begin{conjecture}\label{k4-noc5conj}
\[\pi_{K_4^-}(C_5) = 4-6\left((\sqrt{2}+1)^{1/3}-(\sqrt{2}-1)^{1/3}\right).\]
\end{conjecture}

Given the difference in the proportion of vertices assigned to part $2$ between the case where we are maximising the number of edges and the case where we are maximising the number of copies of $K_4^-$ in a $C_5$-free $3$-graph, one could expect that the way to maximise the number of copies of $4.2$---that is, of $4$-sets spanning exactly $2$ edges---would also be to take a blow-up of $([2], \{112\})$, iterated inside part $2$, with a suitable proportion of vertices (say a little over $1/2$) being assigned to part $2$ at each stage of the iteration. This yields an asymptotic density of only 
\[\max_{x \in [0,1]} \frac{6x^2(1-x)^2}{1-x^4},\]
which is approximately 0.404653. However, it turns out we can do much better using a different construction:

\begin{proposition}\label{4.2noc5}
\[ 0.571428 < 4/7 \leq \pi_{4.2}(C_5) < 0.583852. \]
\end{proposition}

\begin{proof}
The upper bound is from a flag algebra calculation using Flagmatic (see Section~\ref{flagmaticsec} for how to obtain a certificate). For the lower bound, consider a balanced blow-up of $K_4$, iterated inside each part.
\end{proof}

We believe that the lower bound in Proposition~\ref{4.2noc5} is tight:

\begin{conjecture}\label{4.2noc5conj}
\[\pi_{4.2}(C_5)=4/7.\]
\end{conjecture}

While the upper bound we can obtain is still some way off $4/7$, the following exact result gives us rather more confidence about Conjecture~\ref{4.2noc5conj}:

\begin{theorem}\label{4.2noc5f32}
\[\pi_{4.2}(C_5, F_{3,2})=9/16.\]
\end{theorem}

\begin{proof} The upper bound is from a flag algebra calculation using Flagmatic (see Section~\ref{flagmaticsec} for how to obtain a certificate). For a lower bound construction, take a balanced blow-up of $K_4$.
\end{proof}

In a sense Theorem~\ref{4.2noc5f32} tells us that if we do not allow ourselves to use iterated blow-up constructions, then a blow-up of $K_4$ is the best we can do.

This trick of forbidding $F_{3,2}$ when we think an iterated construction is best, but cannot close the gap using the semi-definite method, is often helpful, and we shall use it frequently in this paper. As discussed in detail in~\cite{FRV11}, there are heuristic reasons why one would not expect problems that admit iterated blow-up structures as extremal examples to be easily tackled using the semi-definite method; in many cases it is thus sensible to first study extremal problems in the context of $3$-graphs with independent neighbourhoods.

	\subsection{Forbidding $K_4^-$}

The $3$-graph on four vertices with three edges, $K_4^-$, is the smallest $3$-graph with non-trivial Tur\'an density, both in terms of the number of vertices and the number of edges. Disproving an earlier conjecture of Tur\'an, Frankl and F\"uredi~\cite{FF84} showed that $\pi(K_4^-)\geq 2/7$ by considering a balanced blow-up of $H_6$, iterated inside each of its $6$ parts. Using his semi-definite method, Razborov~\cite{R10} proved upper bounds for $\pi(K_4^-)$ quite close to this value (and small improvements were subsequently given in \cite{BT10} and \cite{FRV11}), leading to the natural conjecture that the construction of Frankl and F\"uredi is in fact best possible:

\begin{conjecture}[Frankl-F\"uredi, Razborov] \label{k4-conj}
\[\pi(K_4^-)=2/7.\]
\end{conjecture}

Should the conjecture be true, one would expect that an iterated blow-up of $H_6$ also maximises the number of induced copies of $4.2$. As in the previous subsection, the semi-definite method is not quite able to close the gap; again we refer the reader to~\cite{FRV11} for a discussion of why iterated blow-up constructions might be `hard' for the method.

\begin{proposition}\label{4.2nok4-}
\[0.558139 < 24/43 \leq \pi_{4.2}(K_4^-) < 0.558378\]
\end{proposition}

\begin{proof} The upper bound is from a flag algebra calculation using Flagmatic (see Section~\ref{flagmaticsec} for how to obtain a certificate). The lower bound is from a balanced iterated blow-up of $H_6$.
\end{proof}

We believe that the lower bound is tight:

\begin{conjecture}\label{4.2nok4-conj}
\[\pi_{4.2}(K_4^-)=24/43.\]
\end{conjecture}

As before, restricting the setting to that of $3$-graphs with independent neighbourhoods helps quite a lot, both for the original Tur\'an problem and for the Tur\'an $4.2$-density problem. In \cite{FRV11} it was proved that $\pi(K_4^-, F_{3,2})=5/18$. The extremal construction, a balanced blow-up of $H_6$, is also extremal for the following problem.

\begin{theorem}\label{4.2nok4-nof32}
\[ \pi_{4.2}(K_4^-,F_{3,2})=5/9. \]
\end{theorem}

\begin{proof} The upper bound is from a flag algebra calculation using Flagmatic (see Section~\ref{flagmaticsec} for how to obtain a certificate). The lower bound is from a balanced blow-up of $H_6$.
\end{proof}

	\subsection{Forbidding $K_4^-$ and $C_5$}

In~\cite{FRV11}, we considered the problem  of forbidding both $K_4^-$ and $C_5$. We have a lower bound of $\pi(K_4^-,C_5) \geq 1/4$ by considering a balanced blow-up of a $3$-edge, with the construction iterated inside each of the $3$ parts; and we gave an upper bound of $\pi(K_4^-, C_5) < 0.251073$ using the semi-definite method, leading us to conjecture that the lower bound is tight:

\begin{conjecture}[\cite{FRV11}]\label{nok4-c5conj}
\[\pi(K_4^-,C_5)=1/4.\]
\end{conjecture}

Another construction yielding the same lower bound is as follows: let $H_7$ be the $6$-regular 3-graph on $7$ vertices
\[H_7=([7], \{124,137,156,235,267,346,457,653,647,621,542,517,431,327\}).\]
This can be thought of as the unique (up to isomorphism) $3$-graph $G$ on $7$ vertices such that for every vertex $x \in V(G)$, the link-graph $G_{x}=(V(G)\setminus\{x\}, \{yz: \ xyz \in E(G)\})$ is the $6$-cycle. Alternatively, $H_7$ can be obtained as the union of two edge-disjoint copies of the Fano plane on the same vertex set
\begin{align*}
F_1&=([7], \{124,137,156,235,267,346,457\}) \textrm{ and }\\
F_2 &=([7], \{653,647,621,542,517,431,327\}),
\end{align*}
as depicted in Figure~\ref{furedifano}. (This elegant perspective is due to F\"uredi.)

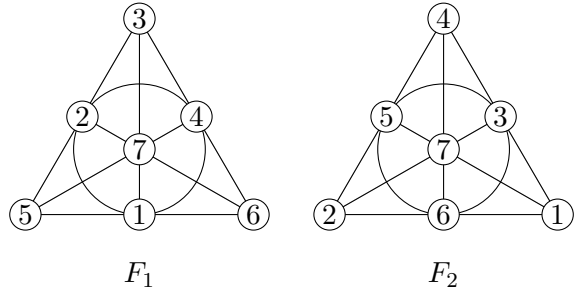
\begin{figure}
\begin{center}
\begin{tikzpicture}
\tikzstyle{blob} = [circle, draw, fill=white, text centered, inner sep=1pt]	
\draw[black] (3/2,sin{60}) circle (sin{60});
\draw[black,fill=black] (0,0) circle (2pt) node[blob] {$5$} --
(3,0) circle (2pt) node[blob] {$6$} --
(3/2,3*sin{60}) circle(2pt) node[blob] {$3$} -- (0,0)
(0,0) -- (30:sqrt{27}/2) circle (2pt) node[blob] {$4$} (3,0) --
+(150:sqrt{27}/2) circle (2pt) node[blob] {$2$}
(3/2,3*sin{60})--(3/2,0) circle (2pt) node[blob] {$1$}
(3/2,sin{60}) circle (2pt) node[blob] {$7$};
\draw (1.5,-0.5) node[below] {$F_1$};
\begin{scope}[xshift=4cm]
\draw[black] (3/2,sin{60}) circle (sin{60});
\draw[black,fill=black] (0,0) circle (2pt) node[blob] {$2$} --
(3,0) circle (2pt) node[blob] {$1$} --
(3/2,3*sin{60}) circle(2pt) node[blob] {$4$} -- (0,0)
(0,0) -- (30:sqrt{27}/2) circle (2pt) node[blob] {$3$} (3,0) --
+(150:sqrt{27}/2) circle (2pt) node[blob] {$5$}
(3/2,3*sin{60})--(3/2,0) circle (2pt) node[blob] {$6$}
(3/2,sin{60}) circle (2pt) node[blob] {$7$};
\draw (1.5,-0.5) node[below] {$F_2$};
\end{scope}
\end{tikzpicture}
\end{center}
\caption{F\"uredi's double Fano construction.} \label{furedifano}
\end{figure}

It is an easy exercise to check that a balanced blow-up of $H_7$ with the construction iterated inside each of the $7$ parts is both $C_5$-free and $K_4^-$-free. (Alternatively, see~\cite{FRV11} for details.) This also gives us a lower bound of $1/4$ on $\pi(K_4^-, C_5)$. When we require independent neighbourhoods, iterated blow-ups are prohibited, and it turns out that a non-iterated blow-up of $H_7$ does better than a blow-up of a $3$-edge (which gives edge-density $2/9$):

\begin{theorem}[\cite{FRV11}]\label{nok4-c5f32}
\[\pi(K_4^-, C_5,F_{3,2})=12/49,\]
with the lower bound attained by a balanced blow-up of $H_7$.
\end{theorem}

Let us now turn to the problem of maximising the number of copies of $4.2$ in a $(C_5, K_4^-)$-free $3$-graph. As we are forbidding $K_4^-$ (which is the same as forbidding $4.3$), one might expect the problem of maximising the density of $4$-sets spanning $2$ edges to be essentially equivalent to the problem of maximising the number of edges. However, the extremal behaviour of the two problems is different. An iterated blow-up of $H_7$ yields a lower bound of $20/57$ ($\approx 0.350877$) for $\pi_{4.2}(K_4^-,C_5)$, but an iterated blow-up of a $3$-edge does much better:

\begin{proposition}\label{4.2nok4-c5}
\[0.461538 < 6/13\leq \pi_{4.2}(K_4^-,C_5) < 0.461645.\]
\end{proposition}

\begin{proof} The upper bound is from a flag algebra calculation using Flagmatic (see Section~\ref{flagmaticsec} for how to obtain a certificate). The lower bound is from a balanced iterated blow-up of a $3$-edge.
\end{proof}

We make the inevitable conjecture that the lower bound in Proposition~\ref{4.2nok4-c5} is tight:

\begin{conjecture}\label{4.2nok4-c5conj}
\[\pi_{4.2}(K_4^-,C_5)=6/13.\]
\end{conjecture}

Besides the relative proximity of the upper and lower bounds in Proposition~\ref{4.2nok4-c5}, further motivation for  Conjecture~\ref{4.2nok4-c5conj} can be found in the following exact result.

\begin{theorem}\label{4.2nok4-c5f32}
\[ \pi_{4.2}(K_4^-,C_5,F_{3,2})=4/9.\]
\end{theorem}

\begin{proof} The upper bound is from a flag algebra calculation using Flagmatic (see Section~\ref{flagmaticsec} for how to obtain a certificate). The lower bound is from a balanced blow-up of a $3$-edge. \end{proof}

By contrast, a balanced blow-up of $H_7$ only gives a lower bound of $120/343$. Thus when $\{K_4^-, C_5, F_{3,2}\}$ is forbidden, the construction that maximises the density of the most dense $3$-graph on four vertices that is allowed, is different to the construction that maximises the edge-density. And it is different in a rather strong sense: not only are the constructions not isomorphic, but there is no homomorphism from $H_7$ into (a blow-up of) a $3$-edge.

Indeed, label the $3$ parts of the blow-up $G$ of a $3$-edge $A$, $B$ and $C$, and suppose $f:\ V(H_7) \rightarrow A\sqcup B\sqcup C$ is a homomorphism. Since $137$ is an edge of $H_7$, it must then be that $1$, $3$ and $7$ are each mapped to different parts $A,B,C$; without loss of generality we may assume that $f(1) \in A$, $f(3) \in B$ and $f(7) \in C$. Since $134$ is also an edge of $H_7$ we must also have $f(4) \in A$. But then $467$ is an edge of $H_7$ with $f(4),f(7)\in C$, and so cannot be mapped by $f$ to an edge of $G$, contradicting our assumption that $f$ is a homomorphism.

This structural difference between the problems of maximising the number of $3$-edges and of maximising the number of copies of $4.2$ in a $K_4^-$-free $3$-graph is a somewhat surprising phenomenon. We ask whether this is due solely to the fact that we are forbidding $C_5$ and $F_{3,2}$ on top of $K_4^-$:

\begin{question}\label{simultextrem}
Let $m$ and $2\leq t \leq \binom{m}{3}$ be integers. Does there exist for every $n \in \mathbb{N}$ an $m.t$-free $3$-graph on $n$ vertices that has both the maximum number of edges and the maximum number of copies of $m.(t-1)$ possible in an $m.t$-free graph?
\end{question}

Of course this question is most interesting when $t=\binom{m}{3}$; here $m.t$ and $m.(t-1)$ consist of just $K_t$ and $K_t^-$ respectively. In this case we believe the answer to Question~\ref{simultextrem} is `yes', which is, in a weaker form, our Conjecture~\ref{simultextrem} from Section~2.1.

	\subsection{Independent neighbourhoods}

We have now seen several examples of how restricting the setting to $3$-graphs with independent neighbourhoods can render Tur\'an problems significantly more tractable to the semi-definite method; we refer the reader to~\cite{FRV11} for a heuristic discussion of why this might be so. In this subsection, we study Tur\'an $H$-density problems in $F_{3,2}$-free $3$-graphs for their own sake. The Tur\'an density problem for $F_{3,2}$ was solved by F\"uredi, Pikhurko and Simonovits:

\begin{theorem}[F\"uredi, Pikhurko, Simonovits~\cite{FPS03}]
\[\pi(F_{3,2})= 4/9.\]
\end{theorem}

In fact, they showed rather more: the unique, stable extremal configuration is an unbalanced blow-up of $([2],\{112\})$, with the size of the two parts chosen so as to maximise the number of edges, so that roughly $2/3$ of the vertices are assigned to part $1$ and $1/3$ to part $2$~\cite{FPS05}. Note that this configuration is $K_4$-free. We therefore expect it to maximise the induced density of $K_4^-$ in an $F_{3,2}$-free graph. This does turn out to be the case, with the minor caveat that we need to change the proportion of vertices in each part; we now want a random $4$-set to have exactly three vertices in part $1$ and one in part $2$, rather than a random $3$-set to have two vertices in part $1$ and one in part $2$. 

\begin{theorem}\label{k4-nof32}
\[\pi_{K_4^-}(F_{3,2})=27/64.\]
\end{theorem}

\begin{proof} The upper bound is from a flag algebra calculation using Flagmatic (see Section~\ref{flagmaticsec} for how to obtain a certificate). The lower bound is from a blow-up of $([2], \{112\})$, with three quarters of the vertices assigned to part $1$ and the rest to part $2$.
\end{proof}

As the above construction is $C_5$-free, Theorem~\ref{k4-nof32} also implies \[\pi_{K_4^-}(C_5, F_{3,2})=27/64, \]
providing us with an analogue for $K_4^-$ of Theorem~\ref{4.2noc5f32} from Section~2.2.

The next $3$-graph whose density in $F_{3,2}$-free $3$-graphs we investigate is $K_4$. Observing that $K_5$ is not $F_{3,2}$-free, one is naturally led to guess that the $K_4$-density is maximised by taking a balanced blow-up of $K_4$. This does indeed turn out to be the case:

\begin{theorem}\label{k4nof32}
\[\pi_{K_4}(F_{3,2})=3/32.\]
\end{theorem}

\begin{proof} The upper bound is from a flag algebra calculation using Flagmatic (see Section~\ref{flagmaticsec} for how to obtain a certificate). The lower bound is from a balanced blow-up of $K_4$.
\end{proof}

Thus we are left with two $3$-graphs on $4$ vertices whose density in $F_{3,2}$-free $3$-graphs we would like to maximise. However, we have been unable to obtain sharp results:

\begin{proposition}\label{4.2nof32}
\[
\begin{array}{rcccl}
4/9 & \leq & \pi_{4.1}(F_{3,2}) & < & 0.514719,\\
9/16 & \leq & \pi_{4.2}(F_{3,2}) & < & 0.627732.
\end{array}
\]
\end{proposition}

\begin{proof} The upper bounds are from flag algebra calculations using Flagmatic (see Section~\ref{flagmaticsec} for how to obtain a certificate). The lower bounds are from balanced blow-ups of $([3], \{112,223,331\})$  and $K_4$ respectively.
\end{proof}

\subsection{Inducibility}

In this subsection, we study $\pi_{H}(\emptyset)$ for small $3$-graphs $H$. The quantity $\pi_{H}(\emptyset)$ is often called the \emph{inducibility} of $H$. Let $\bar G$ denote the \emph{complement} of a 3-graph $G$; that is, the graph containing all edges not present in $G$. A graph $G$ is said to be \emph{self-complementary} if $G$ and $\bar G$ are isomorphic.

It is easy to see that the $H$-density of a $3$-graph $G$ is equal to the $\bar H$-density of $\bar G$. Two immediate consequences of this are:

\begin{lemma} \label{complementarity}
For any $3$-graph $H$,
\[\pi_H(\emptyset) = \pi_{\bar H}(\emptyset).\]
\end{lemma}

\begin{lemma}
If $H$ is self-complementary, then either there are either at least two extremal constructions, or the extremal construction is itself self-complementary.
\end{lemma}

We first study $\pi_{H}(\emptyset)$ for the $3$-graphs $H$ with $\vert V(H) \vert = 4$. Clearly we have $\pi_{K_4}(\emptyset)=\pi_{\bar K_4}(\emptyset)=1$, so this leaves us only two values to determine, $\pi_{4.2}(\emptyset)$ and $\pi_{K_4^-}(\emptyset)$ (which by Lemma~\ref{complementarity} is the same as $\pi_{4.1}(\emptyset)$).

\begin{theorem} \label{4.2}
\[\pi_{4.2}(\emptyset) = 3/4.\]
\end{theorem}

\begin{proof} The upper bound is from a flag algebra calculation using Flagmatic (see Section~\ref{flagmaticsec} for how to obtain a certificate). The lower bound is from the following random geometric construction.

First of all, place $n$ vertices on the boundary of the unit disc, spaced at equal intervals. Each pair of vertices $(x,y)$ defines a chord of the unit circle. Consider the division of the unit disc into polygonal regions given by these chords. We independently assign each region a value $0$ or $1$ with equal probability. Then, for each triple of vertices $(x,y,z)$, we add the 3-edge $xyz$ if and only if the sum of the values of the regions contained inside the triangle $xyz$ is odd. This gives us our construction.

We shall now prove that with positive probability, at least $3/4$ of the 4-sets of vertices induce the graph $4.2$. Let us begin with two observations.

First of all, let $R$ be any collection of regions. Then the probability that the sum of their values is odd is exactly $1/2$. (So in particular, our construction has $3$-edge density $1/2$.) Second, if $R$ and $R'$ are two disjoint collections of regions, the parity of the sum of the values of the regions in $R$ is independent from the parity of the sum of the values of the regions in $R'$.

From now on, let us speak of the parity of a collection of regions as a shorthand for the parity of the sum of the values of the regions it contains. Consider a 4-set of vertices $S=\{a,b,c,d\}$. We may assume without loss of generality that when traversing the unit circle clockwise from $a$, the vertices $b$, $c$ and $d$ are met in that order, as depicted in Figure~\ref{quadrilateral}. So $a,b,c,d$ are the vertices of a convex quadrilateral. Let $e$ be the intersection point of the diagonals $ac$ and $bd$, and let $R_1$, $R_2$, $R_3$ and $R_4$ denote the triangles
$abe$, $bce$, $cde$ and $ade$.

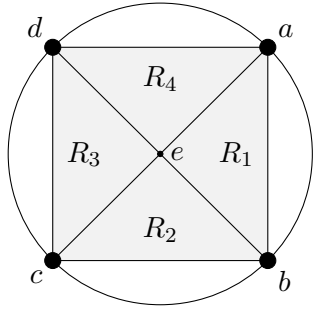
\begin{figure}
\begin{center} \begin{tikzpicture}
\fill[gray!10] (135:2)--(225:2)--(315:2)--(45:2)--cycle;
\draw (0,0) circle (2) (315:2)--(225:2)--(135:2)--(45:2)--(315:2)--(135:2)
(45:2)--(225:2);
\draw[fill=black] (45:2) circle (3pt) node[above right] {$a$}
(315:2) circle (3pt) node[below right] {$b$}
(225:2) circle (3pt) node[below left] {$c$}
(135:2) circle (3pt) node[above left] {$d$}
(0, 0) circle (1pt) node[right] {$e$}
(180:1) node {$R_3$} (90:1) node {$R_4$} (0:1) node {$R_1$} (270:1) node {$R_2$};
\end{tikzpicture} \end{center}
\caption{From the proof of Theorem~\ref{4.2}.} \label{quadrilateral}
\end{figure}

Now, if zero or four of the $R_i$ have odd parity, then the $4$-set $S=\{a,b,c,d\}$ spans no edges in our construction. If one or three of the $R_i$ has odd parity, then $S$ spans two edges; this happens with probability $1/2$. If two of the $R_i$ have odd parity, there are two cases to consider: either the two $R_i$ with odd parity are adjacent to each other---i.e{.} their boundaries intersect in a nontrivial line segment---in which case $S$ spans two edges, or they are opposite one another, in which case $S$ spans four edges. The former case occurs with probability $1/4$. 

Therefore the probability that $S$ spans two edges is $3/4$. Since the choice of $S$ was arbitrary, it follows that with positive probability our construction gives a $4.2$-density of at least $3/4$, whence we are done.
\end{proof}

We note that the lower bound construction is quite different from previously known $3$-graphs constructions. Of those that have appeared in the literature, it resembles most the geometric construction of Frankl and F\"uredi~\cite{FF84}, which it in some sense generalises. This construction also features vertices on the unit circle, where $3$-edges are added whenever the corresponding triangle contains the origin in its interior. (Incidentally, this construction has a $4.2$-density of $1/2$.)

Let us now consider the inducibility of $K_4^-$. Here by contrast we do not believe we have a good lower bound. We get a similar upper bound if we forbid $4$-sets of vertices from spanning exactly one edge.

\begin{proposition}\label{k4-inducibility}
\[0.592592 < 16/27 \leq \pi_{K_4^-}(\emptyset) < 0.651912.\]
Also,
\[16/27 \leq \pi_{K_4^-}(\Text{induced } 4.1) \leq 0.650930.\]
\end{proposition}

\begin{proof} The upper bounds are from flag algebra calculations using Flagmatic (see Section~\ref{flagmaticsec} for how to obtain a certificate). The lower bound in both cases is from Tur\'an's construction: a balanced blow-up of $([3], \{123,112,223,331\})$.
\end{proof}

It seems likely that both $\pi_{K_4^-}(\emptyset)$ and $\pi_{K_4^-}(\Text{induced } 4.1)$ take values close to $0.65$. Since Tur\'an's construction has no induced copies of $4.1$ and is (by Theorem~\ref{k4-nok4}) a $K_4$-free $3$-graph maximising the $K_4^-$-density, this would indicate that the actual extremal construction(s) for the inducibility of $K_4^-$ have strictly positive $K_4$-density.

Turning to $5$-vertex graphs, we are able to obtain a few more exact results.

\begin{theorem}\label{5.6}
\[\pi_{5.4}(\emptyset)=\pi_{5.6}(\emptyset)= 20/27.\]
\end{theorem}

\begin{proof} The upper bound is from a flag algebra calculation using Flagmatic (see Section~\ref{flagmaticsec} for how to obtain a certificate). The lower bound (for $5.6$) is from a balanced blow-up of $([3], \{112, 221, 223, 332, 113, 331\})$. (This is just a balanced tripartition with all $3$-edges meeting a part in two vertices exactly.) 
\end{proof}

\begin{theorem}\label{5.7}
\[\pi_{5.3}(\emptyset) =\pi_{5.7}(\emptyset)= 20/27.\]
\end{theorem}

\begin{proof} The upper bound is from a flag algebra calculation using Flagmatic (see Section~\ref{flagmaticsec} for how to obtain a certificate). The lower bound (for $5.7$) is from a balanced blow-up of $([3], \{111, 222, 333, 123, 112, 223, 331\})$. (This is just Tur\'an's construction with all three parts made complete.)
\end{proof}

\begin{theorem}\label{5.9}
\[\pi_{5.1}(\emptyset)=\pi_{5.9}(\emptyset)=5/8.\]
\end{theorem}

\begin{proof} The upper bound is from a flag algebra calculation using Flagmatic (see Section~\ref{flagmaticsec} for how to obtain a certificate). The lower bound is obtained by taking a complete balanced bipartite $3$-graph.
\end{proof}

In the forthcoming note~\cite{FRV12}, we prove that the complete balanced bipartite $3$-graph is in fact the stable extremum for the inducibility of $5.9$. This relates Theorem~\ref{5.9} to a conjecture of Tur\'an on the Tur\'an density of $K_5$, the complete $3$-graph on $5$ vertices.

\begin{conjecture}[Tur\'an]\label{k5conj}
\[\pi(K_5)=3/4.\]
\end{conjecture}

One of the constructions attaining the bound is given by taking a balanced complete bipartite $3$-graph. Many other non-isomorphic constructions are known~\cite{S95}. However, what Theorem~\ref{5.9} shows is that the complete bipartite $3$-graph is, out of all of these, the one which maximises the number of induced copies of $K_5^-$, that is of $5$-sets spanning all but one of the possible $3$-edges. This is a direct analogue of our earlier result Theorem~\ref{k4-nok4}.

We close this section on $3$-graphs by giving upper bounds on the inducibility of two other $3$-graphs on $5$ vertices.

\begin{proposition}\label{f32}
\[0.349325 < \alpha < \pi_{F_{3,2}}(\emptyset) < 0.349465,\]
where $\alpha$ is the maximum of
\[\frac{10 x^2 (1-x)^3}{1-x^5}\]
in the interval $[0, 1]$.
\end{proposition}

\begin{proof} The upper bound is from a flag algebra calculation using Flagmatic (see Section~\ref{flagmaticsec} for how to obtain a certificate). The lower bound is obtained by taking a unbalanced blow-up of $([2], \{112,222\})$, iterated inside part $1$, where a proportion $\alpha$ of the vertices are assigned to part 1 at each stage.
\end{proof}

We believe that the lower bound construction given above is extremal:

\begin{conjecture}
\[\pi_{F_{3,2}}(\emptyset)=\max_{x \in [0,1]} \frac{10 x^2 (1-x)^3}{1-x^5}.\]
\end{conjecture}

Finally, we note that the random geometric construction given in Theorem~\ref{4.2}, which is extremal for the inducibility of the self-complementary graph $4.2$, also gives a reasonably good lower bound on the inducibility of the self-complementary graph $C_5$:

\begin{proposition} \label{c5}
\[ 0.1875=3/16\leq \pi_{C_5}(\emptyset) < 0.198845.\]
\end{proposition}

\begin{proof} The upper bound is from a flag algebra calculation using Flagmatic (see Section~\ref{flagmaticsec} for how to obtain a certificate). The lower bound comes from considering the random geometric construction we introduced in the proof of Theorem~\ref{4.2}. As the vertices are scattered on the unit circle, any five of them define a convex pentagon. Drawing in the diagonals divides this pentagon into $11$ disjoint regions. The result then follows from a rather tedious case analysis.
\end{proof}

\begin{figure}
\begin{center}
\usetikzlibrary{arrows}
\begin{tikzpicture}
\tikzstyle{vertex} = [circle, draw, fill=white, text centered, inner sep=2pt]
\draw (72*1+90:1.8) node[vertex] (a) {$a$}
(72*2+90:1.8) node[vertex] (b) {$b$}
(72*3+90:1.8) node[vertex] (c) {$c$}
(72*4+90:1.8) node[vertex] (d) {$d$}
(72*5+90:1.8) node[vertex] (e) {$e$};
\draw[-latex] (a)--(b); \draw[-latex] (a)--(c);
\draw[-latex] (d)--(a); \draw[-latex] (d)--(b);
\draw[-latex] (d)--(c); \draw[-latex] (d)--(e);
\draw[dashed] (b)--(c)--(e)--(a);
\end{tikzpicture}
\end{center}
\caption{From the proof of Proposition~\ref{dto3}.} \label{dto3fig}
\end{figure}
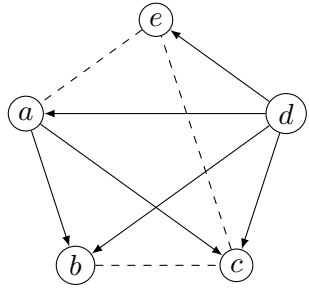

\section{A digression into directed graphs}

\subsection{The out-star of order $3$}
	
We define the \emph{out-star} of order $k$ to be the directed graph
\[\vec{S}_k= ([k], \{\vec{1i}: \ i \in [k]\setminus\{1\}\}). \]
That is, the star with $k-1$ edges oriented away from the centre. In this subsection, we shall be interested in particular in $\vec{S}_3$ and its relation to the $3$-graph $C_5$, the strong cycle on $5$ vertices.

Given a directed graph $D$ on $n$ vertices, let us define a $3$-graph $G(D)$ on the same vertex set by setting $xyz$ to be a $3$-edge whenever the $3$-set $\{x,y,z\}$ induces a copy of $\vec{S}_3$ in $D$.

\begin{proposition}\label{dto3}
$G(D)$ is a $(C_5, K_4)$-free $3$-graph.
\end{proposition}
\begin{proof}
Let us first show that $G(D)$ is $K_4$ free. Suppose $\{a,b,c,d\}$ is a $4$-set of vertices in $G(D)$ that spans a $K_4$. Without loss of generality, we may assume that $\vec{ab}, \vec{ac}$ are in $E(D)$. Therefore neither $\vec{bc}, \vec{cb}$ are in $E(D)$. Since $\{a,b,d\}$ also spans a $3$-edge in $G(D)$, it follows that $\vec{ad} \in E(D)$ and $\vec{bd}, \vec{db}\notin E(D)$. But then $\{b,d,c\}$ spans at most one edge of $D$, and hence cannot be a $3$-edge of $G(D)$, a contradiction.

Now suppose $\{a,b,c,d,e\}$ is a $5$-set of vertices that spans a $C_5$ in $G(D)$, with edges $abc$, $bcd$, $cde$, $dea$ and $eab$. Since $abc$ is an edge, $\{a,b,c\}$ must induce a copy of $\vec{S}_3$ in $D$.

First of all, suppose we have $\vec{ab}, \vec{ac}$ in $E(D)$, and $\vec{bc}, \vec{cb}$ not in $E(D)$, as depicted in Figure~\ref{dto3fig}. As $bcd \in E(G(D))$, $\{b,c,d\}$ must span a copy of $\vec{S}_3$, and we must have $\vec{db}, \vec{dc} \in E(D)$. Similarly, as $cde \in E(G(D))$ we have $\vec{de} \in E(D)$ and $\vec{ec},\vec{ce} \notin E(D)$. Again as $dea \in E(G(D))$ we must have $\vec{da} \in E(D)$ and $\vec{ae}, \vec{ea} \notin E(D)$. But then $\{e,a,b\}$ cannot induce a copy of $\vec{S}_3$ in $D$, and hence $eab$ cannot be a $3$-edge of $G(D)$, a contradiction.

By symmetry, this argument also rules out the possibility of having $\vec{ca}, \vec{cb}$ both in $E(D)$ and $\vec{ab}, \vec{ba} \notin E(D)$. This leaves us with one last possibility, namely that both $\vec{ba}, \vec{bc}$ are in $E(D)$ and neither of $\vec{ac}, \vec{ca}$ is in $E(D)$. Since $bcd$ is an edge of $G(D)$, this implies that $\vec{bd}$ is in $E(D)$ while neither of $\vec{cd}, \vec{dc}$ is. But this also leads to a contradiction by our previous argument, with $bcd$ now playing the role of $abc$. Thus $G(D)$ must be $C_5$-free, as claimed.
\end{proof}

In fact more is true: the proof of the second part of Proposition~\ref{dto3} generalises to show that, for all integers $t\geq 3$ with $t$ congruent to $1$ or $2$ modulo $3$, $G(D)$ contains no copy of the strong $t$-cycle
\[C_{t}=([t], \{ 123, \ 234, \ \dots \ , (t-2)(t-1)t, \ (t-1)t1, \ t12 \}.\]
An interesting question is whether some kind of converse is true. Note that an immediate consequence of Proposition~\ref{dto3} is the following:

\begin{corollary}\label{s3c5ineq}
\[\pi_{\vec{S}_3}(\emptyset) \leq \pi(K_4, C_5, C_7).\]
\end{corollary}

It is easy to check that the conjectured extremal $3$-graph construction of Mubayi and R\"odl for the $\pi(C_5)$ problem is both $K_4$-free and $C_{t}$-free for all $t\geq 3$, where $t$ is congruent to $1$ or $2$ modulo $3$. We ask therefore the following question:

\begin{question}\label{3gtodq}
Does there exist, for every $\varepsilon>0$, a $\delta=\delta(\varepsilon)>0$ and $N=N(\varepsilon)$ such that if $G$ is a $C_5$-free $3$-graph on $n>N$ vertices with at least $(2\sqrt{3}-3 - \delta)\binom{n}{3}$ edges, then there is a directed graph $D$ on $n$ vertices such that the $3$-graphs $G$ and $G(D)$ differ on at most $\varepsilon \binom{n}{3}$ edges?
\end{question}

An affirmative answer to Question~\ref{3gtodq} would, by our next result, automatically imply Conjecture~\ref{c5conj}:

\begin{theorem}\label{1213}
\[\pi_{\vec{S}_3}(\emptyset)= 2\sqrt{3}-3.\]
\end{theorem}

\begin{proof} The upper bound is from a flag algebra calculation using Flagmatic (see Section~\ref{flagmaticsec} for how to obtain a certificate). The lower bound comes from an unbalanced blow-up of the directed graph 
\[\vec{S}_2=([2], \{\vec{12}\}),\]
and iterating the construction inside part $1$, setting at each stage of the construction a proportion $(\sqrt{3}-1)/2$ of the vertices in part $1$ and the remaining $(3-\sqrt{3})/2$ proportion of the vertices in part $2$.
\end{proof}

Denote the lower bound construction in Theorem~\ref{1213} by $D$; then $G(D)$ is exactly the $C_5$-free construction of Mubayi and R\"odl described in Section~2.2. It is an interesting question as to why exactly it is that Flagmatic can give us exact bounds on the $\vec{S}_3$-density problem for directed graphs, but not for the Tur\'an density problem for the $3$-graph $C_5$.

In a forthcoming note~\cite{FRV12}, we use the directed graph removal lemma of Alon and Shapira~\cite{AS04} to prove that the construction $D$ is stable for this problem. 

Theorem~\ref{1213} is, to the best of our knowledge, the first known irrational inducibility. But perhaps more significantly, it is the first `simple' problem for which an iterated blowup construction can be shown to be extremal. Pikhurko~\cite{P11b} has shown the far stronger result that every iterated blowup construction for $3$-graphs is the unique extremal configuration for \emph{some} Tur\'an density problem. However his proof works by a kind of compactness argument, and does not give explicit families of suitable forbidden $3$-graphs, but rather proves that such families exist.

\subsection{Other directed graphs}

Let us now consider $\vec{S}_4$. As in the previous subsection, given a directed graph $D$ we define a $3$-graph $G$ on the same vertex set by letting $xyz$ be a $3$-edge if the $3$-set $\{x,y,z\}$ induces a copy of the out-star of order $3$, $\vec{S}_3$. Then the number of copies of $K_4^-$ in $G(D)$ is exactly the number of copies of $\vec{S}_4$ in $D$, whence we have:

\begin{proposition}
\[\pi_{\vec{S}_4}(\emptyset) \leq \pi_{K_4^-}(C_5).\]
\end{proposition}

\begin{proof} By Proposition~\ref{dto3}, for every directed graph $D$, $G(D)$ is $C_5$-free. The claimed inequality follows directly from our remark that copies of $K_4^-$ in $G(D)$ correspond exactly to copies of $\vec{S}_4$ in $D$.
\end{proof}

We conjectured in Section~2.2 that 
\[\pi_{K_4^-}(C_5)=4-6\left((\sqrt{2}+1)^{1/3}-(\sqrt{2}-1)^{1/3}\right),\]
or, more helpfully, the maximum of 
\[ \frac{ 4 x(1-x)^3}{1-x^4} \]
for $x \in [0,1]$, which is attained at the unique real root of $3t^3+3t^2+3t-1$. We have been unable to prove this using the semi-definite method, but, just as in the previous subsection, the directed graph problem proves to be more tractable, allowing us to show:

\begin{theorem}\label{121314}
\[\pi_{\vec{S}_4}(\emptyset) = \frac{ 4 p(1-p)^3}{1-p^4}, \]
where $p$ is the real root of $3t^3+3t^2+3t-1$.
\end{theorem}

\begin{proof} The upper bound is from a flag algebra calculation using Flagmatic (see Section~\ref{flagmaticsec} for how to obtain a certificate). The lower bound comes from an unbalanced blow-up of $\vec{S}_2$ and iterating the construction inside part $1$, setting at each stage of the construction a proportion $p$ of the vertices in part $1$ and the remaining $1-p$ proportion of the vertices in part $2$.
\end{proof}

As in Theorem~\ref{1213}, call $D$ our lower bound construction for Theorem~\ref{121314}. Then $G(D)$ coincides exactly with our lower bound construction in Section~2.2 for $\pi_{K_4^-}(C_5)$, which we conjectured to be optimal.

\vspace{10pt}

So what about $\pi_{\vec{S}_k}(\emptyset)$ for general $k$? Given Theorems~\ref{1213} and~\ref{121314} it is natural to guess that in general an unbalanced blowup of $\vec{S}_2$ iterated inside part $1$ should be best possible. As we have shown, this is true for the cases $k=3$ and $k=4$, and we conjecture that this remains true for general $k$:

\begin{conjecture}\label{skconj}
For every $k\geq 3$,
\[\pi_{\vec{S}_k}(\emptyset)= \alpha_k,\]
where
\[ \alpha_k = \max_{x \in [0,1]} \frac{ k x(1-x)^{k-1}}{1-x^k}, \]
with the unique stable extremal configuration being a blow-up of $\vec{S}_2$ iterated inside part $1$, with a proportion $\alpha_k$ of the vertices assigned to part $1$ at every iteration. \end{conjecture}

With a little bit of calculus, we can describe $\alpha_k$ more precisely; the maximum of 
\[ \frac{ k x(1-x)^{k-1}}{1-x^k} \]
occurs when $x=x_k$, where $x_k$ is the unique positive root of the polynomial
\[ (k-1)(t + t^2 + \cdots + t^{k-1}) - 1. \]
Note that $x_k \in [0,1/(k-1)]$ and $x_k \rightarrow 1/(k-1)$ as $k\rightarrow \infty$, as we would expect from our construction. Thus also $\alpha_k \rightarrow 1/e$ as $k \rightarrow \infty$.

%Let
%\[f_k(x)= \frac{ k x(1-x)^{k-1}}{1-x^k}.\]
%Then
%\begin{align*}
%\frac{d f_k(x)}{dx}&= \frac{k}{(1-x^k)^2}\left(kx^{k-1}. x(1-x)^{k-1}+ (1-x^k)\left((1-x)^{k-1}-(k-1)x(1-x)^{k-2}\right) \right)\\
%& = \frac{k{(1-x)}^{k-1}}{(1-x^k)^2}\left(kx^k +(1-x^k) -(k-1)\left(1+x+x^2 \ldots +x^{k-1}\right)\right)\\
%&=\frac{k{(1-x)}^{k-1}}{(1-x^k)^2} \left(1-(k-1)x \left(1+x+x^2 + \ldots x^{k-2}\right)\right).
%\end{align*}
%Thus, writing $p_k$ for the unique positive root of the the polynomial 
%\[1-(k-1)x\frac{(1-x^{k-1})}{1-x},\]
%we have $\alpha_k= f_k(p_k)$.

\section*{Acknowledgement}
We would like to thank Colin Reid for coming up with the idea that led to the lower bound construction in Theorem~\ref{4.2}.

\end{document}